\documentclass[a4paper,10pt,reqno]{amsart}
\usepackage{verbatim}
\usepackage{amssymb}      
\usepackage{mathrsfs}       

\usepackage{enumerate}       
\usepackage[active]{srcltx}       
\numberwithin{equation}{section}

\usepackage{t1enc}       
\usepackage[utf8x]{inputenc}

\newtheorem{theorem}{Theorem}[section]        
\newtheorem{lemma}[theorem]{Lemma}       
\newtheorem{problem}[theorem]{Problem}

\newtheorem{corollary}[theorem]{Corollary}       
\newtheorem{observation}[theorem]{Observation}       
\newtheorem{claim}{Claim}[theorem]       
\newtheorem{proposition}[theorem]{Proposition}

\theoremstyle{definition}       
\newtheorem{definition}[theorem]{Definition}       
\theoremstyle{remark}       
\newtheorem{remark}{Remark}         
\newcommand{\mc}[1]{\mathcal{#1}}       
       
\newcommand{\mbb}[1]{\mathbb{#1}}

\newcommand{\setm}{\setminus}       
\newcommand{\empt}{\emptyset}       
\newcommand{\subs}{\subset}       
\newcommand{\dom}{\operatorname{dom}}       
\newcommand{\ran}{\operatorname{ran}}       
\def\<{\left\langle}       
\def\>{\right\rangle}       
\newcommand{\concat}{\mathop{{}^{\frown}\makebox[-3pt]{}}}
   
\newcommand{\paseq}[1]{\mc C^P(#1)}
\newcommand{\plaseq}[2]{\mc C^P_{#1}(#2)}
\newcommand{\lev}[2]{\operatorname{I}_{#1}(#2)}
\newcommand{\aseq}[1]{\mc C(#1)}
\newcommand{\laseq}[2]{\mc C_{#1}(#2)}

\newcommand{\SEQ}{\operatorname{SEQ}}

\newcommand{\almd}[2]{\hat {\mbb A}(#1, #2)}
\newcommand{\indn}[2]{\hat{\mbb I}(#1, #2)}

\newcommand{\om}{\omega}
\newcommand{\ka}{\kappa}
\newcommand{\un}{\bigcup}
\newcommand{\al}{\alpha}
\newcommand{\be}{\beta}
\newcommand{\ga}{\gamma}

\newcommand{\conseq}[2]{\<#1\>_{#2}}

\title{Cardinal sequences of Lindel\"{o}f scattered P-spaces}

\author[J. C. Martinez]{Juan Carlos Mart\'{\i}nez}
\address{Facultat de Matem\`atiques  i Inform\`atica\\ Universitat de Barcelona \\ Gran
 Via 585 \\ 08007 Barcelona, Spain}
\email{jcmartinez@ub.edu}
\author[L. Soukup]{
Lajos Soukup }
\address{HUN-REN Alfr{\'e}d R{\'e}nyi Institute of Mathematics }
\email{soukup@renyi.hu}
      
\thanks{The research on and preparation of this paper was
supported by  OTKA grant   K129211}       
\date{\today}       
       
\begin{document}       
       
\begin{abstract}       
  We continue our investigation of cardinal sequences associated with locally Lindelöf, 
  scattered, Hausdorff P-spaces (abbreviated as LLSP spaces). We outline a method for 
  constructing LLSP spaces from cone systems and partial orders with specific properties. 
  Additionally, we establish limitations on the cardinal sequences of LLSP spaces. 
  Finally, we present both a necessary condition and a distinct sufficient condition 
  for a sequence $\<{\kappa}_{\alpha}:{\alpha}<{\omega}_2\>$ 
  to be the cardinal sequence of an LLSP space.
\end{abstract}       
       
\maketitle

\section{Introduction}

Recall that if  $X$ is a topological space and $\al$ is an ordinal, the $\al\mbox{th}$-{\em Cantor-Bendixson level} of $X$ is defined  as $I_{\al}(X) =$ the set of isolated points of $X\setminus \bigcup \{I_{\be}(X) : \be < \al\} $.  Then, $X$ is {\em scattered} if there is an ordinal $\al$ such that $X = \bigcup \{I_{\be}(X) : \be < \al\} $. If $X$ is a scattered space, the {\em cardinal sequence} of $X$ is defined as $\mbox{CS}(X) = \langle |I_{\al}(X)| : \al < \delta \rangle$, where $\delta$ is the least ordinal $\ga$ such that $I_{\ga}(X)$ is a finite set. Recall that a topological space $X$ is a {\em P-space}, if the intersection of every countable family of open sets in $X$ is open in $X$.''

In this paper, we continue our investigation of cardinal sequences of LLSP spaces, building upon our work initiated in \cite{MaSoLinSP22}. Let us recall that, 
in the context of this study, an LLSP space refers to a locally Lindelőf, scattered, Hausdorff P-space.

Cone systems and partial orders with special properties have played a fundamental role in the construction of locally compact scattered Hausdorff (LCS) spaces. 
These concepts have been explored in works such as \cite{JuSo-countably-tight-96} and \cite{Ba-locally-generic-02}. 
In Section \ref{sc:cone}, we extend a similar analysis to LLSP spaces instead of LCS spaces.

%There are many known and easily provable restriction concerning cardinal sequence of LCS space. 
%In Section \ref{sc:basic} we try to prove similar theorems for LLSP spaces.  

There are numerous well-established and easily demonstrable 
constraints regarding the cardinal sequences of LCS spaces. In Section \ref{sc:basic}, 
we endeavor to establish analogous theorems for LLSP spaces.

To simplify the presentation of our results, we use $\paseq {\alpha}$ and $\aseq{{\alpha}}$ to represent, respectively, 
the classes of all cardinal sequences of length $\alpha$ 
associated with LLSP and LCS spaces.

In their work \cite{JW}, Juhász and Weiss 
gave a full description of the class $\aseq{{\omega}_1}$ by 
establishing  
that a  sequence 
$\<{\kappa}_{\alpha}:{\alpha}<{\omega}_1\>$ of infinite cardinals
is the cardinal sequence of an LCS space iff 
${\kappa}_{\beta}\le ({\kappa}_{\alpha})^{\omega}$ holds for each 
${\alpha}<{\beta}<{\omega}_1$.

In Section \ref{sc:short}, we will encounter the challenge of 
characterizing short sequences for LLSP spaces, 
which seems to be a much harder task. 
Initially,  we demonstrate  that the natural  transfer of the previously 
mentioned result of Juhász and Weiss   
for LLSP spaces is unsuccessful. 
In Theorem \ref{tm:positive}, we provide both a necessary condition and a 
distinct sufficient condition for a sequence $\<{\kappa}_{\alpha}:{\alpha}<{\omega}_2\>$ 
to be the cardinal sequence of an LLSP space.

The necessary and sufficient conditions mentioned earlier 
coincide in the Cohen model (see Theorem \ref{tm:Cohen}). However, in Theorem \ref{tm:con}, we illustrate 
that these properties  consistently differ from each other.
As it turns out, we lack a complete ZFC characterization of $\paseq{3}$, 
the class of cardinal sequences associated with LLSP spaces of height  $3$.

Let $\conseq \kappa\alpha $ denote the constant $\kappa$-valued sequence of length $\alpha$. 
It is almost trivial that if
$\conseq \kappa\alpha{}^\frown \<{\lambda}\>\in \aseq{{\alpha}+1}$ with $cf({\alpha})\le {\omega}$,
then ${\lambda}\le {\kappa}^{\omega}$. Analogously, one can easily prove that if 
$\conseq \kappa\alpha{}^\frown \<{\lambda}\>\in \paseq{{\alpha}+1}$ then 
${\lambda}\le {\kappa}^{\omega_1}$ 
provided that  either ${\alpha}$ is a successor ordinal or  
$cf({\alpha})={\omega}_1$. However, the inequality ${\lambda}\le {\kappa}^{\omega_1}$  does not necessary hold when $cf({\alpha})={\omega}$:
in Theorem \ref{Theorem 1} we prove that 
it is consistent that $2^{\omega_1} = \omega_2$ and $\langle \omega_2\rangle_{\om} \concat \langle \omega_3\rangle \in \paseq {\omega + 1}$.

\section{LLSP spaces from cone systems}\label{sc:cone}

\begin{definition}
Given a scattered space $X$ we say that a 
 neighborhood assignment $U:X\to {\tau}_X$ is an {\em LLSP neighborhood assignment}
iff 
\begin{enumerate}[(1)]
\item $U(y)$ is a clopen Lindelöf neighborhood of $y$ for each $y\in X$,
\item if $y\in I_{\alpha}(X)$, then 
\begin{displaymath}
U(y)\setm I_{<\alpha}(X)=\{y\}.
\end{displaymath} 
\end{enumerate}
\end{definition}
As we observed in \cite[Proposition 1.1]{MaSoLinSP22},  LLSP spaces are 0-dimensional, and so 
the following holds: 
\begin{proposition}\label{pr:llspassignment}
Every LLSP space $X$ has an LLSP neighborhood assignment
$U_X$.
\end{proposition}

If $U:X\to \mathcal P(X)$ is a function and $I\subs X$, write 
\begin{displaymath}
U[I]=\bigcup\{U(x):x\in I\}.
\end{displaymath}

\begin{definition}\label{df:llsp-cone}
Assume that $X$ is a set and $U:X\to \mathcal P(X)$ is a function.
We say that $U$ is a {\em weak LLSP-cone assignment (for $X$)}    iff for some ordinal 
${\delta}$ the set $X$ has  a partition $\<X_{\alpha}:{\alpha}<{\delta}\>$ and 
there is a function $D:{[X]}^{2}\to {[X]}^{{\omega}}$
such that: 
\begin{enumerate}[(a)]
\item \label{cone:a} $\forall {\beta}<{\delta}$ $\forall p\in X_{\beta}$\ 
($p\in U(p)$ and $U(p)\setm \{p\}\subs X_{<{\beta}}$),\smallskip
\item \label{cone:c} for each $\{x,y\}\in {[X]}^{2}$,
if $x\notin U(y)$ and $y\notin U(x)$, then  
$D(x,y)\subs  {U(x)\cap U(y)}$ and 
\begin{displaymath}
U(x)\cap U(y)\subs U[D(x,y)],
\end{displaymath}
\item \label{cone:d}  for each $\{x,y\}\in {[X]}^{2}$
if $x\in U(y)$, then  
$D(x,y)\subs {U(x)\setm  U(y)}$  and  
\begin{displaymath}
U(x)\setm U(y)\subs U[D(x,y)].
\end{displaymath}
\end{enumerate}
We say that  $U$ is an {\em LLSP-cone assignment (for $X$)} if 
\eqref{cone:b} below also holds: 
\begin{enumerate}[(a)]\addtocounter{enumi}{3}
  \item \label{cone:b} $\forall {\alpha}<{\beta}<{\delta}$ $\forall p\in X_{\beta}$ $\forall I\in {[U(p)\setm\{p\}]}^{{\omega}}$
  \begin{displaymath}
  (U(p)\setm U[I])\cap X_{\alpha}\ne \empt,
  \end{displaymath}
\end{enumerate}
We will say that the partition $\<X_{\alpha}:{\alpha}<{\delta}\>$ is 
a {\em witnessing partition} for $U$,
and the function $D$ is the {\em LLSP-function} associated with $U$.

If $\al < \delta$, we write $X_{\leq \al} = \bigcup \{X_{\be} : \be \leq \al \}$.
\end{definition}

The following statement is straightforward:
\begin{proposition}
  If  $U$ is a weak  LLSP-cone assignment, then a witnessing  partition $\<X_{\alpha}:{\alpha}<{\delta}\>$
is  determined by the following recursive formula:
\begin{displaymath}
X_{\beta}=\{p\in X\setm X_{<{\beta}}: U(p)\setm X_{<{\beta}}=\{p\} \},
\end{displaymath}
and ${\delta}=\min\{{\beta}:X_{\beta}=\empt\}$.
 
\end{proposition}

\begin{definition}
  For $U:X\to \mathcal P(X)$ and $p\in X$ let
  \begin{displaymath}
  \mathcal B_{U}(p)=\{U(p)\setm U[I]:I\in {[U(p)\setm \{p\}]}^{{\omega}}\},
  \end{displaymath}
  and 
  \begin{displaymath}
  \mathbb B_U=\{\mathcal B_U(p):p\in X\}.
  \end{displaymath}
  \end{definition}

  \begin{theorem}\label{tm:llspgen}\par (1) 
   If $U:X\to \mathcal P(X)$ is a weak  LLSP-cone assignment with 
    a witnessing partition
    $\<X_{\alpha}:{\alpha}<{\delta}\>$, then 
    \begin{enumerate}[$(i)$]
    \item $\mathbb B_U$ is a 
    neighborhood system of an LLSP space $X_U=\<X,{\tau}_U\>$, 
    \item 
    $U(p)$ is a clopen Lindelőf subspace of $X_U$ for each $p\in X$,
    \item % such that
    $I_{\le \alpha}(X_U)\supset X_{\le\alpha}$ for each ${\alpha}<{\delta}$.
    \end{enumerate} 
    Moreover, if $U$ is a LLSP-cone assignment, then 
\begin{enumerate}[(i)]
\item[(iv)] $I_{\alpha}(X_U)= X_{\alpha}$ for each ${\alpha}<{\delta}$.
\end{enumerate}

    \par (2) If $X=\<X,{\tau}\>$ is an LLSP-space,  
 and   
$U:X\to \tau$ is a neighborhood assignment such that
for each  ${\alpha}<{\delta}$ and for each   $p\in I_{\alpha}(X)$ 
\begin{displaymath}
\tag{$\dag$}
\text{$U(p)$ is a clopen Lindelöf subset of $X$
with $U(p)\setm I_{<{\alpha}}(X)= \{p\}$,}
\end{displaymath}
then $U$ is an LLSP-cone assignment  with the witnessing partition
$\<I_{\alpha}(X):{\alpha}<{\delta}\>$,  and ${\tau}_U={\tau}$. 
    \end{theorem}

    \begin{corollary}\label{cr:llsptocone}
    $X=\<X,{\tau}\>$ is an LLSP-space iff 
    there is an LLSP-cone assignment $U:X\to \mc P(X)$
    such that ${\tau}={\tau}_U$.  
      \end{corollary} 
    
\begin{proof}
If $U:X\to \mc P(X)$ is an LLSP-cone assignment, then 
$X=\<X,{\tau}_U\>$ is an LLSP-space by Theorem \ref{tm:llspgen}(1).

On the other hand, if $X$ is an LLSP-space, then 
by Proposition \ref{pr:llspassignment},  there is a 
  neighborhood assignment $U:X\to {\tau}_X$ such that 
  if $p\in I_{\alpha}(X)$ then 
  $U(p)$ is a clopen Lindelöf subset and 
  \begin{displaymath}
  U(p)\setm \{p\}\subs I_{<{\alpha}}(X).
  \end{displaymath}
Then ${\tau}={\tau}_U$ by Theorem \ref{tm:llspgen}(2).
\end{proof}

    \begin{proof}[Proof of Theorem \ref{tm:llspgen}]
      First, we prove (1).   Assume  $U:X\to \mathcal P(X)$ is an LLSP-cone assignment with 
      a witnessing partition
      $\<X_{\alpha}:{\alpha}<{\delta}\>$.
      \begin{claim}
          $\mathbb B_U$ is a 
          neighborhood system of a topology.    
      \end{claim}
      
      \begin{proof}
      By \cite[Proposition 1.2.3]{En-Gen-Top}, we should check that 
      \begin{enumerate}[(BP1)]
      \item $\forall p\in X$  $\mathcal B_U(p)\ne \empt$ and $p\in \bigcap \mathcal B_U(p)$,
      \item $\forall p\in X$ $\forall q\in V\in \mathcal B_U(p)$ $\exists W\in \mathcal B_U(q)$ $W\subs V$,
      \item $\forall p\in X$ $\forall V_0,V_1\in \mathcal B_U(p)$ $\exists  V\in \mathcal B_U(p)$ $V\subs V_0\cap V_1$.
      \end{enumerate}
      First observe that (BP1) and (BP3) are straightforward from the construction.
      
      To check (BP2) assume that $q\in U(p)\setm U[I]\in \mathcal B_U(p)$, where  
      $I\in {[U(p)\setm \{p\}]}^{{\omega}}$.
      
      Then there is $D(q,p)\in {[U(q)\setm U(p)]}^{{\omega}}$ such that
      $U(q)\setm U(p)\subs U[D(q,p)]$
      and for each $i\in I$ there is $D(q,i)\in {[U(q)\setm \{q\}]}^{{\omega}}$ such that
      $U(q)\cap U(i)\subs U[D(q,i)]$. 
      Then 
      \begin{multline*}
      U(q)\setm (U(p)\setm U[I])\subs (U(q)\setm U(p))\cup (U(q)\cap U[I])\subs \\
      (U(q)\setm U(p))\cup \bigcup_{i\in I}(U(q)\cap U(i))\subs U[D(q,p)]\cup 
      \bigcup_{i\in I} U[D(q,i)]
      =\\ U[D(q,p)\cup \bigcup_{i\in I}D(q,i)].
      \end{multline*}
     Thus, writing $J=D(q,p)\cup \bigcup_{i\in I}D(q,i)$ we have   
      \begin{displaymath}
      U(q)\setm U[J]\subs U(p)\setm U[I],
      \end{displaymath} 
      and $U(q)\setm U[J]\in \mathcal B_U(q)$.
      
      So $\mathbb B_U$ is a  neighborhood system of a topology ${\tau}_U$.
      \end{proof}  
      \begin{claim}
          $\<X,{\tau}_U\>$ is $T_2$.
      \end{claim}
      
      \begin{proof}
      Assume that $\{x,y\}\in {[X]}^{2}$.
      
      If $x\notin U(y)$ and $y\notin U(x)$, then 
      $U(x)\cap U(y)\subs U[D(x,y)]$, and  
      $U_x=U(x)\setm U[D(x,y)]\in \mathcal B_U(x)$
      and 
      $U_y=U(y)\setm U[D(x,y)]\in \mathcal B_U(y)$
      and $U_x\cap U_y=\empt$.
      
      If $x\in U(y)$, then 
      $U(y)\setm U(x)\in \mathcal B_U(y)$, $U(x)\in \mathcal B_U(x)$,
and $(U(y)\setm U(x))\cap U(x)=\empt.$
      \end{proof}

      The topology  ${\tau}_U$ is clearly  a $P$-space topology.

        \begin{claim}
      
       For each ${\alpha}<{\delta}$ and for each $p\in X_{\alpha}$ the subspace
        $U(p)$ is Lindelöf.    
       \end{claim}
      
      \begin{proof}[Proof]
      We prove it by transfinite induction on ${\alpha}$.
      
      Assume that  $U(q)$ is Lindelöf for each ${\beta}<{\alpha}$ and $q\in X_{\beta}$.
      
      Let $\mathcal V$ be an open cover of $U(p)$. Pick $V\in \mathcal V $ with $p\in V$.
      Since $\mathcal B_U(p)$ is a neighborhood base of $p$, we have $U(p)\setm U[I]\subs V$ for some $I\in {[U(p)\setm \{p\}]}^{{\omega}}$.
      Then $U(p)\setm V\subs U[I]$, and $U[I]$ is Lindelöf because it is the union of 
      countable many Lindelöf subsets. So $U(p)\setm V$ is also Lindelöf because it is 
      a closed subset of the Lindelöf space $U[I]$. Thus, 
      there is a countable $\mathcal W\in {[\mathcal V]}^{{\omega}}$
      such that $U(p)\setm U(q)\subset \bigcup \mathcal{W}$.
      Thus, $U(p)\subs V\cup \bigcup \mathcal{W}$. 
      
      Since $\mathcal V$ was arbitrary, we proved that $U(p)$ is Lindelöf.
      \end{proof}
      
      \begin{claim}
       $X_U$ is scattered and $I_{\le\alpha}(X)\supset X_{\alpha}$.
       Moreover, if $U$ is an LLSP-cone assignment, then    
       $I_{\alpha}(X)= X_{\alpha}$.
      \end{claim}
      
      \begin{proof}
      We prove the statement by induction on ${\alpha}$.
      
      Assume that $I_{\le \beta}(X)\supset X_{\beta}$ for ${\beta}<{\alpha}$.
      
      If $p\in X_{\alpha}$, then $U(p)\setm \{p\}\subs X_{<{\alpha}}\subs I_{<{\alpha}}(X)$
      by \ref{df:llsp-cone}\eqref{cone:a} and by the inductive assumption,
      so $p$ is isolated in $X\setm I_{<{\alpha}}(X)$, and so 
      $p\in I_{\alpha}(X)$.

Assume now that       $U$ is an LLSP-cone assignment and 
$I_{\beta}(X) =X_{\beta}$ for ${\beta}<{\alpha}$.
  If $p\in X_{\gamma}$ for some ${\gamma}>{\alpha}$ and 
      $I\in [U(p)\setm \{p\}]^{\omega}$, then 
      $(U(q)\setm U[I])\cap X_{<\alpha}\ne \empt$ by \ref{df:llsp-cone}\eqref{cone:b}, so 
      $p$ is not isolated 
      in $X\setm X_{<{\alpha}}=X\setm I_{<{\alpha}}(X)$, so 
      $p\notin I_{\alpha}(X)$. 
      
      So $I_{\alpha}(X)=X_{\alpha}$.
    \end{proof}

      \medskip Now, we prove (2).
      So, assume that $X$ and $U$ satisfy the hypothesis of (2). Then, we show that 
      $U$ is an LLSP-cone assignment with the witnessing partition 
      $\<I_{\alpha}(X):{\alpha}<{\delta}\>$.
      Property \eqref{cone:a} follows from $(\dag)$.
      
      To check \eqref{cone:b}
      assume that 
      ${\alpha}<{\beta}<{\delta}$,  $p\in X_{\beta}$  and 
      $I\in {[U(p)\setm\{p\}]}^{{\omega}}$. 
      
      Then $U(p)\setm U[I]=\bigcap_{i\in I}(U(p)\setm U(i))$,
      so $G=U(p)\setm U[I]$ is an open set because $X$ is a P-space, and clearly 
       $p\in G$. Since $I_{\alpha}(X)$ is dense in $X\setm I_{<{\alpha}}(X)$, it follows
       that $G\cap I_{\alpha}(X)\ne \empt$.  
      
       To check \eqref{cone:c}  assume that $\{x,y\}\in {[X]}^{2}$
       with  $x\notin U(y)$ and $y\notin U(x)$.
       
       Then 
       \begin{displaymath}
       U(x)\cap U(y)\subs\bigcup\{U(i): i\in U(x)\cap U(y)\}.
       \end{displaymath}
       Since $U(x)\cap U(y)$ is Lindelöf, there is a countable   
      $D(x,y)\in {[U(x)\cap U(y)]}^{{\omega}}$ such that 
       \begin{displaymath}
       U(x)\cap U(y)\subs U[D(x,y)].
      \end{displaymath}

      To check \eqref{cone:d}  assume that $\{x,y\}\in {[X]}^{2}$
      with  $x\in U(y)$. 
      
      Then 
      \begin{displaymath}
      U(x)\setm  U(y)\subs\bigcup\{U(i): i\in U(x)\setm  U(y)\}.
      \end{displaymath}
      Since $U(x)\setm U(y)$ is Lindelöf, there is a countable 
      $D(x,y)\in {[U(x)\setm U(y)]}^{{\omega}}$ such that 
      \begin{displaymath}
      U(x)\setm U(y)\subs U[D(x,y)].
      \end{displaymath}
      
      Finally, we should show that ${\tau}_U={\tau}$.
      
      Since ${\tau}$ is a P-space topology, 
      $U(p)\setm U[I]$ is open for each 
      $p\in X$ and $I\in {[U(p)\setm \{p\}]}^{{\omega}}$, and so 
      ${\tau}_U\subs {\tau}$.
      
      We know that ${\tau}$ is 0-dimensional, so it is enough to show 
      that if $G\in {\tau}$ is clopen and Lindelöf, then $G\in {\tau}_U$.
      
      For each $p\in G$,
      \begin{displaymath}
      U(p)\setm G\subs \bigcup\{U(i):i \in U(p)\setm G\},
      \end{displaymath}
      so there is a countable subset $D(p)\in {[U(p)\setm G]}^{{\omega}}$
      such that 
      \begin{displaymath}
          U(p)\setm G\subs \bigcup\{U(i):i \in D_p\}.
          \end{displaymath}
       Then $U(p)\setm U[D_p]\in \mathcal U_B(p)$ and $U(p)\setm U[D_p]\subs G$.
       Since $p$ was arbitrary, we proved that $G$ is open in ${\tau}_U$,
       which completes the proof.

So we proved theorem \ref{tm:llspgen}.
\end{proof}

\begin{comment}

{ \begin{definition}\label{df:llsp-poset}
We say that a poset    $\<X,\preceq\>$ is an 
  {\em {LLSP}-poset} if  the function
  $U_\preceq:X\to \mc P(X)$ defined by the formula 
  $U_\preceq (p)=\{q\in X:q\preceq p\}$ is a LLSP-cone assignment,
  i.e. 
  iff for some ordinal 
${\delta}$ the set $X$ has  a partition $\<X_{\alpha}:{\alpha}<{\delta}\>$
such that 
\begin{enumerate}[(a)]
\item \label{poset:a} 
if $\{x,y\}\in {[X]}^{2}$,  $x\prec y$, $x\in X_{\alpha}$ and $y\in X_{\beta}$, then ${\alpha}<{\beta}$,
\item \label{poset:b} 
if ${\alpha}<{\beta}<{\delta}$, $x\in X_{\beta}$ and $I\in 
{[\{y:y\prec x\}]}^{ {\omega}}$
then there is $z\in X_{\alpha}$ such that $z\prec x$ 
and $z\not\prec y$ for each $y\in I$,
\item \label{poset:c} for each $\{x,y\}\in {[X]}^{2}$
there is  $D(x,y)\in {[X]}^{{\omega}}$ such that 
for each $t\in X$
\begin{displaymath}
\text{($t\preceq x$ and $t\preceq y$) iff $t\preceq z$ for some $z\in D(x,y)$.}
\end{displaymath}
\end{enumerate}   
  \end{definition}

\begin{remark}
If $x\in U_{\preceq}(y)$, then $U_\preceq(x)\subs U_{\preceq}(y)$, 
and so $U_\preceq(x)\setm U_{\preceq}(y)=\empt$, so  
property  (\ref{cone:d}) from  definition \ref{df:llsp-cone} trivially holds for $U_\preceq$.  
\end{remark}

}

\end{comment}

\section{Basic properties of cardinal sequences of LLSP-spaces}
\label{sc:basic}

\subsection*{Reduction theorem}

Let us recall that $\paseq {\alpha}$
 and $\aseq {\alpha}$
denote the class of all cardinal sequences of
length $\alpha$ associated with LLSP and  LCS spaces, respectively.
For any cardinal  ${\kappa}$ write 
\begin{align*}
\plaseq{{\kappa}}{{\alpha}}=&\{s\in \paseq{{\alpha}}:
{\kappa}=s(0)=\min \ran(s)\},\\
\laseq{{\kappa}}{{\alpha}}=&\{s\in \aseq{{\alpha}}:
{\kappa}=s(0)=\min \ran(s)\}\\
\end{align*} 
In \cite[Theorem 2.1]{JSW} the following statement was proved:

\noindent {\bf Theorem.} {\em 
For any ordinal ${\alpha}$ we have
  $f\in \aseq {\alpha}$ iff for some natural number $n$ there is a
  decreasing sequence $\lambda_0>\lambda_1>\dots>\lambda_{n-1}$ of
  infinite cardinals and there are ordinals ${\alpha}_0,\dots,
  {\alpha}_{n-1}$ such that
  ${\alpha}={\alpha}_0+\cdots+{\alpha}_{n-1}$ and $f=f_0\concat \
  f_1\concat \cdots \ \concat f_{n-1}$ with
  $f_i\in\laseq{\lambda_i}{\alpha_i}$ for each $i < n$.
}

We will prove a similar statement for LLSP-spaces as well

\begin{theorem}\label{mfact} For any ordinal ${\alpha}$ we have
  $f\in \paseq {\alpha}$ iff for some natural number $n$ there is a
  decreasing sequence $\lambda_0>\lambda_1>\dots>\lambda_{n-1}$ of
  infinite cardinals and there are ordinals ${\alpha}_0,\dots,
  {\alpha}_{n-1}$ such that
  ${\alpha}={\alpha}_0+\cdots+{\alpha}_{n-1}$ and $f=f_0\concat \
  f_1\concat \cdots \ \concat f_{n-1}$ with
  $f_i\in\plaseq{\lambda_i}{\alpha_i}$ for each $i < n$.
  \end{theorem}

  \begin{proof}
  We first prove that the condition is necessary, so fix $f\in\paseq
  {\alpha}$. Let us say that $\beta < \alpha$ is a {\em drop point} in
  $f$ if for all $\gamma < \beta$ we have $f(\gamma) > f(\beta)$.
  Clearly $f$ has only finitely many, say $n,$ drop points; let
  $\{\beta_i : i < n\}$ enumerate all of them in increasing order. (In
  particular, we have then $\beta_0 = 0.$) For each $i < n$ let us set
  $\lambda_i = f(\beta_i)$, then we clearly have
  $\lambda_0>\lambda_1>\dots>\lambda_{n-1}.$
  
  For each $i < n-1$
  let $\alpha_{i} = \beta_{i+1} - \beta_i $ be the unique ordinal such
  that $\beta_{i}+\alpha_{i}= \beta_{i+1}$,
  moreover define the sequence $f_i$ on $\alpha_i$ by the stipulation
  $$f_i(\xi) = f(\beta_i + \xi)$$ for all $\xi < \alpha_i.$
  Similarly, let $\alpha_{n-1} = \alpha - \beta_{n-1}$
  be the unique ordinal such that $\beta_{n-1}+\alpha_{n-1}=\alpha,$
  and define $f_{n-1}$ on $\alpha_{n-1}$ by the stipulation
  $$f_{n-1}(\xi) = f(\beta_{n-1} + \xi)$$ for all $\xi < \alpha_{n-1}.$
  Now, it is obvious that we have $f=f_0\concat \ f_1\concat \cdots \
  \concat f_{n-1}$. Moreover,  $f_i\in\plaseq{\lambda_i}{\alpha_i}$ for each
  $i < n$. Indeed, writing
  \begin{displaymath}
  Y_i=\bigcup_{{\beta}_i\le {\xi}<{\beta}_{i+1}}I_{\xi}(X)
  \end{displaymath} 
  we have $f_i=\SEQ(Y_i)$, 
and $Y_i$ is an LLSP-space because it is a closed subset of the open
subspace $I_{<{\beta}_{i+1}}(X)$ of $X$.

  We shall now prove that the condition is also sufficient. In fact,
  this will follow from the next lemma.
  
  \begin{lemma}\label{pr:conc}
    If  $f\in \paseq {\alpha}$, $g\in \paseq {\beta}$, and
    $f({\nu})\ge g(0)$ for each ${\nu}<{\alpha}$ then
    $f{}\concat  g\in \paseq {{\alpha}+{\beta}}$.
    \end{lemma}

    Indeed, given the sequences $f_i$ for $i < n$, this lemma enables us
to inductively define for every $i=0,\dots,n-1$ an LCS space $Z_i$
with cardinal sequence $f_0\concat \cdots \concat f_i$ because
\begin{displaymath}
f_i(0)={\lambda}_i<{\lambda}_{i-1}=\min\{f_j({\nu}):j<i, {\nu}<{\alpha}_j\}.
\end{displaymath}
In particular, then we have $\SEQ(Z_{n-1})=f_0\concat \ f_1\concat
\cdots \ \concat f_{n-1}$.
\end{proof}

    \begin{proof}[Proof of lemma \ref{pr:conc}]
    Let $Y$ be an LLSP space with cardinal sequence $g$ and satisfying
    $\lev {\beta}Y= \emptyset$. Next
     fix LLSP spaces
     $X_y$ for all
    $y\in \lev {0}{Y}$, each having the cardinal sequence $f$ and
    satisfying $ \lev {\alpha}{X_y}=\empt$. Assume also that the
    family $\{Y\}\bigcup\{X_y : y \in \lev {0}Y\}$ is disjoint.

     We then define
    the space $Z=\<Z,{\tau}\>$ as follows. 
    \begin{displaymath}
      Z=Y\cup\bigcup\{X_y: y\in I_0(Y) \}.
      \end{displaymath}

    Let $U_Y$ be an LLSP-cone assignment for $Y$,
    and $U_p$ be an LLSP-cone assignment 
    for $X_p$ for 
    $p\in \lev 0Y$.

    Define $U:Z\to \mc P(Z)$ as follows:
    \begin{displaymath}
    {U(z)}=\left\{\begin{array}{ll}
    {U_p(z)}&\text{if $z\in X_p$}\\\\
    {U_Y(z)\cup \bigcup\{X_p:p\in \lev 0{U_Y(z)}\} }&\text{if $z\in Y$}
    \end{array}\right.
    \end{displaymath}
    
    Then $U$ is an LLSP-cone assignment.
    
    Consider the space $X_U$.
    
    Then $\bigcup_{p\in \lev 0Y}X_p$ is dense open in $X_U$, so 
    for ${\zeta}<{\alpha}$ we have 
    \begin{displaymath}
    \lev {\zeta}{X_U}=\bigcup\{\lev {\zeta}{X_p}:p\in \lev 0Y\}
    \end{displaymath} 
    
    So $\lev {<{\alpha}}{X_U}=\bigcup \{X_p:p\in \lev 0Y\}$.
    
   Thus, for ${\zeta}<{\beta}$ 
    \begin{displaymath}
    \lev {{\alpha}+{\zeta}}{X_U}=\lev {\zeta}Y.
    \end{displaymath}
    So $SEQ(X_U)=f\concat g$.
    
    \end{proof}

\subsection*{Sequence restrictions}

In  \cite[Lemmas 1-3.]{JW}
the authors proved the following statement:

\noindent {\bf Theorem.} {\em 
Assume that $s\in \aseq {\beta}$.
\begin{enumerate}[(1)]
\item   $|{\beta}|\le {2^{s(0)}}$
and $s({\alpha})\le 2^{s(0)}$ for each ${\alpha}<{\beta}$.
\item    if  ${\alpha}+1<{\beta}$ then
$s({\alpha}+1)\le s({\alpha})^{\omega}$.
\item    If  $\delta<\beta$ is a limit
ordinal, and $C$ is any cofinal subset of $\delta$, then
\[
s({\delta})\le\prod\{s(\alpha):\alpha\in C\}\ .
\]
\end{enumerate}
}

We can prove the following version of these statements for LLSP spaces.

\begin{theorem}\label{tm:restrictions}
  Assume that $s\in \paseq {\beta}$.
\begin{enumerate}[(1)]
\item   $|{\beta}|\le {2^{s(0)}}$
and $s({\alpha})\le 2^{s(0)}$ for each ${\alpha}<{\beta}$.
\item   \label{Lemma1}  if  ${\alpha}+1<{\beta}$ then
$s({\alpha}+1)\le s({\alpha})^{\omega_1}$.
\item   \label{Lemma2} If  $\delta<\beta$ is a limit
ordinal with $cf({\delta})>{\omega}$ and $C$ is any cofinal subset of $\delta$, then
\[
s({\delta})\le\prod\{s(\alpha):\alpha\in C\}\ .
\]
\end{enumerate}
\end{theorem}

\begin{proof}
  Assume that $s$ is the cardinal sequence of the LLSP-space $X$.

  \noindent 
(1) Since  $X$ is regular, we have  
$|X|\le 2^{d(X)}=2^{I_0(X)}=2^{s(0)}$.
Since ${\beta}=ht(X)\le |X|$ and $s({\alpha})= |I_{\alpha}(X)|\le |X|$ for each ${\alpha}<ht(X)$,
we proved (1).

\medskip
\noindent (2)
Let $U$ be an LLSP-cone assignment of $X$. 
If $x\in I_{{\alpha}+1}(X)$, then $$\{x\}=\lev{{\alpha}+1}{X}
\cap \overline{(U(x)\cap I_{\alpha}(x))},$$
and so $|U(x)\cap I_{\alpha}(x)|\ge {\omega}_1$.

Pick $Y_x\in {[U(x)\cap I_{\alpha}(x))]}^{{\omega}_1} 
$ for each $x\in I_{{\alpha}+1}(X)$. 

If $\{x,y\}\in [I_{{\alpha}+1}(X)]^2$, 
then  $|U(x)\cap U(y)|\le {\omega}$ because $U(x)\cap U(y)$ is Lindelöf, 
and so $Y_x\ne Y_y$.
Thus, $|\lev{{\alpha}+1}X|\le |I_{\alpha}(X|^{{\omega}_1}$.

\medskip
\noindent (3)
We can assume that $C$ has order type $cf({\delta})$.
%If $cf({\delta})\ge {\omega}_1$, then the statement is easy:
Let $U$ be an LLSP-cone assignment of $X$. 

For each $x\in \lev{\delta}X$ pick a function 
$f_x$ with $\dom(f_x)=C$ such that 
$f_x({\gamma})\in \lev{{\gamma}}X\cap U(x)$. 
Then $\{x\}=I_{\delta}(X)\cap \overline{\ran f_x}$, 
and we have at most $\prod\{s({\alpha}):{\alpha}\in C\}$ many $f_x$
functions. 
\end{proof}

\begin{remark}
In Theorem \ref{Theorem 1} we show that   without assuming $cf({\delta})>{\omega}$
  (3) is not true   even in the following weaker form :  
  \begin{enumerate}[(1)]
    \item[(3)']  \label{Lemma2-2} If  $\delta<\beta$ is a limit
    ordinal and $C$ is any cofinal subset of $\delta$, then
    \[
    s({\delta})\le\Big(\prod\{s(\alpha):\alpha\in C\}\Big)^{{\omega}_1}\ .
    \]
    \end{enumerate}
   \end{remark}

%   without assuming $cf({\delta})>{\omega}$.

%   Using countable conditions starting from a model of GCH  one can force an LLSP space   $X$
%   with cardinal sequence $\<{\omega}_1\>_{\omega}\concat \<{\omega}_2\>$.
%   \end{remark}

% \begin{remark}
% (3) is not true without assuming $cf({\delta})>{\omega}$.
% Using countable conditions starting from a model of GCH  one can force an LLSP space   $X$
% with cardinal sequence $\<{\omega}_1\>_{\omega}\concat \<{\omega}_2\>$.
% \end{remark}

% Maybe (3) holds in the following form: 

% \begin{enumerate}[(1)]
%   \item[(3)']  \label{Lemma2-2} If  $\delta<\beta$ is a limit
%   ordinal and $C$ is any cofinal subset of $\delta$, then
%   \[
%   s({\delta})\le\Big(\prod\{s(\alpha):\alpha\in C\}\Big)^{{\omega}_1}\ .
%   \]
%   \end{enumerate}

Indeed, by Theorem \ref{Theorem 1}, it is  consistent that 
  $2^{\omega_1} = \omega_2$ and $\langle \omega_2\rangle_{\om} \concat \langle \omega_3\rangle \in \paseq {\omega + 1}$. 

\begin{problem}\label{pr:w2}
%(1) Is it consistent that $2^{{\omega}_2}={\omega}_3$ and 
%$\<{\omega}_3\>_{{\omega}}\concat \<{\omega}_4\>\in \paseq{{\omega}+1}$? 
%
%
%\noindent
%(2) 
Is it provable in ZFC that $\<2^{{\omega}_1}\>_{{\omega}}\concat \<(2^{{\omega}_1})^+\>\in \paseq{{\omega}+1}$?
\end{problem}

\section{Short sequences}\label{sc:short}

In \cite{JW}, Juhász and  Weiss characterized the elements of $\aseq{{\omega}_1}$.

\noindent{\bf Theorem. }{\em 
A sequence $\<{\kappa}_{\alpha}:{\alpha}<{\omega}_1\>$ of infinite cardinals
is the cardinal sequence of an LCS space iff 
${\kappa}_{\beta}\le ({\kappa}_{\alpha})^{\omega}$ for each 
${\alpha}<{\beta}<{\omega}_1$.}

The straightforward adaptation of this theorem for LLSP 
spaces leads to the following naive conjecture:
{\em A sequence $\<{\kappa}_{\alpha}:{\alpha}<{\omega}_2\>$ of uncountable 
cardinals
is the cardinal sequence of an LLSP space iff 
${\kappa}_{\beta}\le ({\kappa}_{\alpha})^{\omega_1}$ for each 
${\alpha}<{\beta}<{\omega}_2$.}

This statement  is not provable in ZFC as we will see soon.

To formulate our results we need some new notion.
\begin{definition}\label{df:AI}
If ${\omega}\le {\lambda}\le {\kappa}$ are cardinals, let
\begin{displaymath}
\almd{{\kappa}}{{\lambda}}=\min\{{\mu}:\forall  \mc A\in 
\left[{[{\kappa}]}^{{\lambda}}\right]^{\mu} \exists 
\{A_0,A_1\}\in {[\mc A]}^{2}\ |A_0\cap A_1|={\lambda}
\}.
\end{displaymath}
%\begin{displaymath}
%  \almd{{\kappa}}{{\lambda}}=\sup\{|\mc A|: \mc A\subs  
%  {[{\kappa}]}^{{\lambda}}\land  
%  \forall \{A_0,A_1\}\in {[\mc A]}^{2}\ |A_0\cap A_1|={\lambda}
%  \}.
%  \end{displaymath}
  and 
    \begin{displaymath}
    \indn{{\kappa}}{{\lambda}}=\min\{{\mu}:\forall  \mc I\in 
    \left[{[{\kappa}]}^{\le {\lambda}}\right]^{\mu}\  \exists 
    A\in \mc I\  \exists \mc J\in {[\mc I\setm \{A\}]}^{{\omega}} 
    \ A\subs \bigcup \mc J
    \}.
    \end{displaymath}
%    \begin{displaymath}
%      \indn{{\kappa}}{{\lambda}}=\sup\{|\mc I|:  \mc I\subs 
%      {[{\kappa}]}^{{\lambda}}\land \ \forall 
%      A\in \mc I\  \forall \mc J\in {[\mc I\setm \{A\}]}^{{\omega}} 
%      \ A\not\subs \bigcup \mc J
%      \}.
%      \end{displaymath}
  \end{definition}

If $cf({\lambda})>{\omega}$, then 
clearly $\almd{{\kappa}}{{\lambda}}\le \indn{{\kappa}}{{\lambda}}$.

\begin{observation}\label{obs:upper}
If $\<{\kappa},{\mu}\>$ is the cardinal sequence of an LLSP space, 
then ${\mu}<\almd{{\kappa}}{{\omega}_1}$. 
\end{observation}

\begin{proof}
    Let $U$ be an LLSP neighborhood assignment on $X$.     
    For each $y\in I_1(X)$ pick $A(y)\in {[I_0(x)\cap U(y)]}^{{\omega}_1}$.
    Then $\{A(y):y\in I_1(X) \}\subs {[I_0(X)]}^{{\omega}_1}$ and 
    $A(z)\cap A(y)\subs U(z)\cap U(y)$ is countable for $\{y,z\} \in [I_1(X)]^2$
    because $U(z)\cap U(y)$ is a  discrete, Lindelöf space.  So 
    $|I_1(X)|< \almd{{\kappa}}{{\omega}_1}$. 
\end{proof}

The following lemma is well-known and easy to prove from the Erd\H os-Rado partition theorem 
$(2^{{\omega}})^+\to ({\omega}_1)^2_{\omega}$    
(see Baumgartner \cite{Ba-76}):

\begin{lemma}\label{lm:Ainccc}
If $V\models $ ``$2^{{\omega}_1}={\omega}_2$'' and $P$ satisfies c.c.c. then 
$V^P \models $ $\almd{{\omega}_1}{{\omega}_1}={\omega}_3$.
\end{lemma}

So if  GCH holds in the ground model, and we add ${\omega}_3$ Cohen reals to it, 
then in the generic extension $2^{\omega}=2^{{\omega}_1}={\omega}_3$ +
$\almd{{\omega}_1}{{\omega}_1}={\omega}_3$, so 
$\<{\omega}_1, 2^{{\omega}_1}\>$ is not the cardinal sequence of 
an LLSP space. Hence, the naive conjecture fails. 
Instead of the naive conjecture, we can prove  the following  weaker statements: 

\begin{theorem}\label{tm:positive}
Assume that $\vec {\kappa}=\<{\kappa}_{\alpha}:{\alpha}<{\omega}_2\>$ is a sequence of uncountable 
cardinals.

\noindent (1) If ${\kappa}_{\beta}<\almd{{\kappa}_{\alpha}}{{\omega}_1}$ for each 
${\alpha}<{\beta}<{\omega}_2$, then $\vec{\kappa}$
 is the cardinal sequence of an LLSP space.

%  \noindent (2) If $\vec {\kappa}$ is the cardinal sequence of an LLSP space, then  
%  ${\kappa}_{\beta}< \almd {{\kappa}_{\alpha}}{\omega_1}$ for each 
%  ${\alpha}<{\beta}<{\alpha}+{\omega}<{\omega}_2$.

 \noindent (2) If $\vec {\kappa}$ is the cardinal sequence of an LLSP space, then  
 ${\kappa}_{\beta}< \indn {{\kappa}_{\alpha}}{{\kappa}_{\alpha}}$ for each 
 ${\alpha}<{\beta}<{\omega}_2$.
\end{theorem}

\begin{proof}[Proof of Theorem \ref{tm:positive}(1)]
Let us say that a sequence $\<{\kappa}_{\alpha}:{\alpha}<{\omega}_2\>$
is  {\em good} iff ${\kappa}_{\beta}<\almd{{\kappa}_{\alpha}}{{\omega}_1}$ for each 
${\alpha}<{\beta}<{\omega}_2$.

  \begin{lemma}
    If every good sequence $\vec {\kappa}$ with 
    ${\kappa}(0)=\min\{{\kappa}({\alpha}):{\alpha}<{\omega}_2\}$
    is the cardinal sequence of an LLSP space, then (1) holds. 
\end{lemma}    
\begin{proof}
Let $\vec {\kappa}$ be a good sequence.

Define the sequence 
$0={\alpha}_0<{\alpha}_1<\dots<{\alpha}_{n+1}={\omega}_2$
as follows.

  Let %${\mu}_0={\kappa}_0$ and 
${\alpha}_0=0$
and
\begin{displaymath}
{\alpha}_{i+1}=\min
(\{{\alpha}\in {\omega}_2\setm {\alpha}_i:{\kappa}_{{\alpha}_{i+1}}<{\kappa}_{{\alpha}_i}\}\cup\{{\omega}_2\})
\end{displaymath}
We stop the construction  if ${\alpha}_{n+1}={\omega}_2$.
Since ${\kappa}_{{\alpha}_{i+1}}<{\kappa}_{{\alpha}_i}$ if ${\alpha}_{i+1}<{\omega}_2$,
we will stop in some step   $n\in {\omega}$.  

For $0\le i\le n$
write ${\mu}_i={\kappa}_{{\alpha}_i}$ and  let 
\begin{displaymath}
s_i=\<{\mu}_i\>_{{\alpha}_i}\concat \vec {\kappa}\restriction [{\alpha}_i, {\alpha}_{i+1}).
\end{displaymath}
Observe that $s_i(0)={\mu}_i={\kappa}_{{\alpha}_i}=\min\{s_i({\gamma}):{\gamma}\in \dom(s_i)\}$.
So, by the assumption on the lemma, there is 
an LLSP space   $X_i$ with cardinal sequence $s_i$.
Then $X$, the disjoint union of $\{X_i:i\le n\}$ has cardinal sequence $\vec {\kappa}$. 
\end{proof}

From now on  
we  assume that ${\kappa}_0=\min_{{\alpha}<{\omega}_2}{\kappa}_{\alpha}$.

Write ${\kappa}={\kappa}_0$.

\begin{lemma}\label{lm:cof}
$\sup\{{\kappa}_{\xi}:{\xi}<{\omega}_2\}<\almd{{\kappa}}{{\omega}_1}$.
\end{lemma}

\begin{proof} 
If ${\kappa}>{\omega}_1$, then  
$\operatorname{cf}(\almd{{\kappa}}{{\omega}_1})>{\kappa}\ge {\omega}_2$, so the lemma holds. 

If ${\kappa}={\omega}_1$, then observe that 
$\almd{{\omega}_1}{{\omega}_1}$ is a regular cardinal and clearly
$\almd{{\omega}_1}{{\omega}_1}\ge {\omega}_3$. 
\end{proof}

By Lemma \ref{lm:cof}, we can fix an almost disjoint family $\{A^{\xi}_{\alpha}:{\xi}<{\omega}_2, {\alpha}<{\kappa}_{\xi}\}\subs 
{[{\kappa}]}^{{\omega}_1}$.

It was proved in \cite[Section 2]{MaSoLinSP22} that for every ordinal $\alpha < \omega_3$
 there is an LLSP space with cardinal sequence 
 $\langle \omega_1 \rangle_{\alpha}$.  Then, for each $\sigma\in \kappa$ 
 let $L_{\sigma}$ be an LLSP space with cardinal sequence 
 $\langle \omega_1 \rangle_{\omega_2}$ such that 
 $\{L_{\sigma} : \sigma \in \kappa \}$ is pairwise disjoint.  
 For each $\sigma < \kappa$ let $V_{\sigma}$ be an LLSP-cone assignment
  for $L_{\sigma}$. For each $\sigma < \kappa$ and  $\nu < \omega_2$ pick a point $x^{\sigma}_{\nu}\in I_{\nu}(L_{\sigma})$, 
and let $B^{\sigma}_{\nu}= V_{\sigma}(x^{\sigma}_{\nu})$.

For each ${\xi}<{\omega}_2$ let $I_{\xi}\subs {\xi}$ be a cofinal
set with  $tp(I_{\xi})=cf({\xi}).$  In particular, $I_{{\zeta}+1}=\{{\zeta}\}$
for each ${\zeta}<{\omega}_2$.

Partition every $A^{\xi}_{\alpha}$ into $|I_{\xi}|$-many uncountable sets:
\begin{displaymath}
A^{\xi}_{\alpha}={\bigcup}^*\{A^{\xi}_{{\alpha},{\nu}}:{\nu}\in I_{\xi}\}
\end{displaymath}
and $|A^{\xi}_{\alpha,{\nu}}|={\omega}_1$.

Define the function $i_{{\xi},{\alpha}}:A^{\xi}_{\alpha}\to I_{\xi}$
such that ${\sigma}\in A^{\xi}_{{\alpha}, i_{{\alpha},{\xi}}({\sigma})}$
for each  ${\sigma}\in A^{{\xi}}_{\alpha}$. 

The underlying set of our space will be 
\begin{displaymath}
X=\bigcup_{{\sigma}<{\kappa}}L_{\sigma}\cup
\{\<{\xi},{\alpha}\>:{\xi}<{\omega}_2,{\alpha}<{\kappa}_{\xi}\}.
%\bigcup_{{\xi}<{\omega}_2}\big(\{{\xi}\}\times {\kappa}_{\xi}\big).
\end{displaymath}
For ${\xi}<{\omega}_2$ write 
\begin{displaymath}
X_{\xi}=\bigcup_{{\sigma}\in {\kappa}}\lev{\xi}{L_{\sigma}}
\cup (\{{\xi}\}\times {\kappa}_{\xi})
\end{displaymath}
We will define our topology in such a way that every $L_{\sigma}$ is clopen,
and $X_{\xi}=\lev{\xi}X$.

To do so we need to define an LLSP-cone assignment  
$U:X\to \mc P(X)$  with the witnessing partition 
$\<X_{\zeta}:{\zeta}<{\omega}_2\>$.

We define $U$ as follows:
\begin{enumerate}[(1)]
\item if $y\in L_{\sigma}$, then 
$$
U(y)=V_{\sigma}(y).
$$
\item 
for  $y=\<{\xi},{\alpha}\>$ we define $U(y)$ as follows:
\begin{displaymath}
U(\<{\xi},{\alpha}\>)=\{\<{\xi},{\alpha}\>\}\cup\bigcup \{
  B^{{\sigma}}_{\nu}: {\sigma}\in 
A^{\xi}_{{\alpha},{\nu}}, {\nu}\in I_{\xi}\}.
\end{displaymath}
\end{enumerate}

\begin{lemma}
$U$ is an LLSP-cone assignment  with the witnessing partition 
$\<X_{\zeta}:{\zeta}<{\omega}_2\>$.
\end{lemma}
\begin{proof}
The  only non-trivial part is to check \ref{df:llsp-cone}(\ref{cone:c})
for $\<{\xi},{\alpha}\>\ne \<{\xi}',{\alpha}'\>\in X$.
But in this case we have  
\begin{multline*}
  U(\<{\xi},{\alpha}\>)\cap U(\<{\xi}',{\alpha}'\>)=\\
  \bigcup\{ B^{{\sigma}}_{\nu}: {\sigma}\in 
A^{\xi}_{{\alpha},{\nu}}, {\nu}\in I_{\xi}\}\cap 
\bigcup\{ B^{{\sigma}}_{\nu}: {\sigma}\in 
A^{\xi'}_{{\alpha'},{\nu}}, {\nu}\in I_{\xi'}\}=\\
\bigcup\{B^{\sigma}_{i_{{\xi},{\alpha}}({\sigma})}\cap 
B^{\sigma}_{i_{{\xi}',{\alpha}'}({\sigma})}:
{\sigma}\in A^{\xi}_{{\alpha}}\cap A^{\xi'}_{{\alpha'}}\},
\end{multline*}
and $A^{\xi}_{{\alpha}}\cap A^{\xi'}_{{\alpha'}}$ is countable. 
\end{proof}

So $U$ is an LLSP-cone assignment. 
\end{proof}

\begin{proof}[Proof of Theorem \ref{tm:positive}(2)]
  Assume that $X$ is an LLSP space   with cardinal sequence $\<{\kappa}_{\alpha}:{\alpha}<{\omega}_2\>$  and 
  let $U$ be  an LLSP-cone assignment on $X$ with witnessing partition 
  $\<I_{\alpha}(X):{\alpha}<{\omega}_2\>$.
Assume that ${\alpha}<{\beta}<{\omega}_2$. Let 
\begin{displaymath}
\mc A=\{U(x)\cap I_{\alpha}(X):x\in I_{\beta}(X)\}.
\end{displaymath}
Then $|\mc A|={\kappa}_{\beta}$ and $\mc A$ witnesses that ${\kappa}_{\beta}< \indn {{\kappa}_{\alpha}}{{\kappa}_{\alpha}}$. 
\end{proof}

After proving Theorem \ref{tm:positive} the following questions  raise naturally: 

\begin{enumerate}[(Q1)]
\item Is it true that 
$\almd{{\kappa}}{{\omega}_1}=\indn{{\kappa}}{{\omega}_1}$ for each 
uncountable cardinal ${\kappa}$?
\item  Assume that $\vec {\kappa}=\<{\kappa}_{\alpha}:{\alpha}<{\mu}\>$ is a 
 sequence of uncountable cardinals. 
Is it true that 
$\vec {\kappa}$ is the 
cardinal sequence of an LLSP space iff 
${\kappa}_{\beta}<\indn{{\kappa}_{\alpha}}{{\omega}_1}$ for each
${\alpha}<{\beta}<{\mu}$? 
\end{enumerate}

% \begin{problem}
% (1) Assume that $\vec {\kappa}=\<{\kappa}_{\alpha}:{\alpha}<{\mu}\>$ is a 
% cardinal sequence of an LLSP space. 
% Is it true that ${\kappa}_{\beta}<\indn{{\kappa}_{\alpha}}{{\omega}_1}$ for each
% ${\alpha}<{\beta}<{\mu}$? 

% \noindent 
% (2)Assume that $\vec {\kappa}=\<{\kappa}_{\alpha}:{\alpha}<{\omega}_2\>$ 
% is a sequence of uncountable cardinals such that ${\kappa}_{\beta}<\indn{{\kappa}_{\alpha}}{{\omega}_1}$ for each
% ${\alpha}<{\beta}<{\omega}_2$. Is there an LLSP-space with cardinal 
% sequence $\vec {\kappa}$?
% \end{problem}

% {\em Is it true that 
% $\almd{{\kappa}}{{\omega}_1}=\indn{{\kappa}}{{\omega}_1}$ for each 
% uncountable cardinal ${\kappa}$?}

% If this conjecture is true then a sequence $\<{\kappa}_n:n<{\omega}\>$ of infinite cardinals is 
% the cardinal sequence of an LLSP-space iff ${\kappa}_n<\almd{{\kappa}_m}{{\omega}_1}$
% for each $m<n<{\omega}$.

In this section we address the first question.  
In the next section we give a consistent negative answer to the second question. 

First we show that 
the Cohen model does not separate $\almd{{\omega}_1}{{\omega}_1}$ and $\indn{{\omega}_1}{{\omega}_1}$.

\begin{theorem}\label{tm:Cohen}
If $V\models GCH$ and ${\kappa}$ is a cardinal, then 
\begin{displaymath}
V^{Fn({\kappa},2)}\models \indn{{\omega}_1}{{\omega}_1}={\omega}_3.
\end{displaymath}
\end{theorem}

\begin{proof}
    $\indn{{\omega}_1}{{\omega}_1}\ge {\omega}_3$ holds in ZFC. 

  Assume that 
  $p \Vdash \{A_{\alpha}:{\alpha}<{\omega}_3\}\subs {[{\omega}_1]}^{{\omega}_1}$
  
  For each ${\alpha}<{\omega}_3$ and ${\xi}<{\omega}_1$ let 
  $C({\alpha},{\xi})\subs Fn({\kappa},2)$ be a maximal antichain such that 
  for each $q\in C({\alpha},{\zeta})$
  \begin{displaymath}
  q\Vdash {\xi}\in A_{\alpha} \lor
  q\Vdash {\xi}\notin A_{\alpha}.
  \end{displaymath}

Let $S_{\alpha}=\bigcup\{\dom(q):q\in 
\bigcup_{{\xi}<{\omega}_1}C({\alpha},{\xi})\}$.

By standard $\Delta$-system arguments there is $I\in {[{\omega}_3]}^{{\omega}_3}$ such that: 
\begin{enumerate}[(1)]
\item $\{S_{\alpha}:{\alpha}\in I\}$ i s a $\Delta$-system with kernel $S$,
\item $\forall {\alpha}<{\beta}\in I$ 
$S<S_{\alpha}\setm S<S_{\beta}\setm S$
and $tp(S_{\alpha})=tp(S_{\beta})$,
\item if ${\rho}_{{\alpha},{\beta}}$ denotes the natural bijection between 
 $S_{\alpha}$ and $S_{\beta}$, then $\forall {\xi}\in {\omega}_1$
 \begin{displaymath}
 C({\beta},{\xi})=\{{\rho}_{{\alpha},{\beta}}(q):q\in C({\alpha},{\xi})\},
 \end{displaymath}
 where ${\rho}(q)$ is defined as follows:
 $\dom({\rho}(q))={\rho}''\dom(q)$ and ${\rho}(q)({\rho}({\nu}))=q({\nu})$.
 \item if ${\alpha},{\beta}\in I$, ${\xi}<{\omega}_1$ and $q\in C({\alpha},{\xi})$,
 then \begin{displaymath}
 q\Vdash {\xi}\in A_{\alpha} \text{ iff }  {\rho}_{{\alpha},{\beta}}(q)\Vdash {\xi}\in A_{\beta}. 
 \end{displaymath} 
\end{enumerate}
Pick ${\delta}\in I$ and $J=\{{\alpha}_n:n\in {\omega}\}\in {[I\setm \{{\delta}\}]}^{{\omega}}$.

If $r\le p$ and $r\Vdash {\xi}\in A_{\delta}$, 
then there is $q\in C({\alpha},{\xi})$ such that $q\Vdash {\xi}\in A_{\delta}$ and 
$r\|q$. Let $n\in {\omega}$ such that 
$(S_{{\alpha}_n}\setm S)\cap \dom(r)=\empt$.
Let $s={\rho}_{{\delta},{\alpha}_n}(q)$. Then $s\Vdash {\xi}\in A_{{\alpha}_n}$.
Then $s\restriction S= q\restriction S$ 
so $t=q\cup r\cup s\in P$ and 
$t\le q$ and $t\Vdash {\xi}\in A_{{\alpha}_n}$.

So we proved Theorem \ref{tm:Cohen}.
\end{proof}

\begin{theorem}\label{tm:con}
    It is  consistent that $2^{\omega}=2^{{\omega}_1}={\omega}_3$,
    and 
    \begin{displaymath}
    {\omega}_3=    \almd{{\omega}_1}{{\omega}_1}<
    \indn{{\omega}_1}{{\omega}_1}={\omega}_4.
    \end{displaymath}
    
    \end{theorem}
    
\begin{proof}
The next lemma is the heart of the proof of the theorem. 
\begin{lemma}\label{lm:Ilarge}
    If $V\models GCH$ then there is a c.c.c poset $P$ such that 
    $V^P \models $  $2^{\omega}=2^{{\omega}_1}={\omega}_3$ and  $\indn{{\omega}_1}{{\omega}_1}={\omega}_4$.
    \end{lemma}

 \begin{proof}
 Write ${\mu}={\omega}_3$ and  
let 
\begin{displaymath}
L=\{{\zeta}<{\omega}_1:\forall {\eta}<{\zeta}\ ({\eta}+{\eta})<{\zeta}\},
\end{displaymath}
and for ${\zeta}\in L$ put 
\begin{displaymath}
\mc B_{\zeta}=\{B\subs {\mu}: tp(B)<{\zeta}\}.
\end{displaymath}

Observe that $\mc B_{\zeta}$ is closed under finite union. 

Fix a partition $\{C_{\zeta}:{\zeta}\in L\}$ of ${\omega}_1$ into 
uncountable pieces. Define  the function ${\zeta}:{\omega}_1\to L$
such that ${\delta}\in C_{{\zeta}({\delta})}$.

After this preparation define the poset $P=\<P,\le\>$ as follows.

Let $p=\<D,\{E_{\delta}:{\delta}\in D\}, \{F_{\delta}:{\delta}\in D\}\>\in P$
iff
\begin{enumerate}[(1)]
\item $D\in {[{\omega}_1]}^{<{\omega}}$,
\item  $E_{\delta}\in {[{\mu}]}^{<{\omega}}$, %$E_{\delta}\in {[C_{{\zeta}({\delta})}]}^{<{\omega}}$
\item $F_{\delta}\in \mc B_{{\zeta}({\delta})}$,
\item  $F_{\delta}\cap E_{\delta}=\empt$  for each ${\delta}\in D$.
\end{enumerate} 
For $p\in P$ write 
$$p=\<D^p,\{E^p_{\delta}:{\delta}\in D^p\}, \{F^p_{\delta}:{\delta}\in D^p\}\>.$$
Let  
\begin{align*}
  p\le q \text{ iff }& D^p\supset D^q,\text{ and}\\ 
  &\text{$E^p_{\delta}\supset E^q_{\delta}$
  and $F^p_{\delta}\supset F^q_{\delta}$
  for each ${\delta}\in D^q$.} 
  \end{align*}

Given a generic filter 
$\mc G\subs P$,  
for ${\delta}<{\omega}_1$ and 
${\alpha}<{\mu}$   let
\begin{displaymath}
E_{\delta}=\bigcup\{E^p_{\delta}:p\in \mc G, {\delta}\in D^p\},
\end{displaymath} 
and 
\begin{displaymath}
A_{\alpha}=\{{\delta}<{\omega}_1: {\alpha}\in E_{\delta}\}.
\end{displaymath}

We will show that the family $\{A_{\alpha}:{\alpha}<{\mu}\}$ witnesses that
$\indn{{\omega}_1}{{\omega}_1}>{\omega}_3$.

Clearly, for each $p\in P$,
\begin{displaymath}
p\Vdash (\forall {\delta}\in D^p)\ {\delta}\in 
\bigcap_{{\alpha}\in E^p_{\delta}}A_{\alpha}\setm \bigcup_{{\beta}\in F^p_{\delta}}A_{\beta}.
\end{displaymath}

The following Claim is straightforward from the definition of $P$, and from the fact that 
every $\mc B_{\zeta}$ is closed under finite union for 
 ${\zeta}\in L$.
\begin{claim}\label{cl:comp}
  Two conditions $p$ and $q$ are compatible iff
  $$(E^p_{\delta}\cap F^q_{\delta})=(F^p_{\delta}\cap E^q_{\delta})=\empt$$
  for each ${\delta}\in D_p\cap D_q$.
  \end{claim}

  The next Claim is the key of the proof. 
  \begin{claim}
  $P$ satisfies c.c.c.  
  \end{claim}
  
\begin{proof}
In order to prove the Claim we will apply the following theorem of Erdős and Specker, \cite{ErSp61},  on the existence of free subsets.

\smallskip  
\noindent{\bf 
Free Subset Theorem of Erdős and Specker.}  %\fbox{REFERENCE MISSING}
{\em If ${\eta}<{\omega}_1$ and  $f:{\omega}_1\to {[{\omega}_1]}^{{\omega}}$ is a function  
such that  $tp(f({\alpha}))<{\eta}$ for each ${\alpha}<{\omega}_1$,
then there is an uncountable $f$-free subset of ${\omega}_1$. 
}
\smallskip

Let $\{p_i:i<{\omega}_1\}\subs P$. Write $p_i=\<D^i,\{E^i_{\delta}:{\delta}\in D^i\}, 
\{F^i_{\delta}:{\delta}\in D^i\}\>$ for $i<{\omega}_1$.

By  thinning out our sequence we can assume that 
$\{D^i:i<{\omega}_1\}$ forms a $\Delta$-system
 with kernel $D$.

By Claim \ref{cl:comp}, $p_i$ and $p_j$ are compatible 
iff $(E^i_{\delta}\cap F^j_{\delta})=(F^i_{\delta}\cap E^j_{\delta})=\empt$
for each ${\delta}\in D$.

So we can assume that $D_i=D$ for each $i<{\omega}_1$, as well.

We can also assume that for each ${\delta}\in D$
there is a natural number $k_i$ such that $|E^i_{\delta}|=k_i$ for each $i<{\omega}_1$.

Write $D=\{{\delta}_m:m<n\}$.
Let $I_0={\omega}_1$.
By finite induction we can find  $I_{m+1}\in {[I_m]}^{{\omega}_1}$
such that 
\begin{displaymath}
E^i_{{\delta}_m}\cap F^j_{{\delta}_m}=\empt \text{ for each }\{i,j\}\in {[I_{m+1}]}^{2}.
\end{displaymath}
by applying the Free Subset Theorem $k_i$-times.

After that construction, the elements of $\{p_i:i\in I_{n}\}$ are pairwise 
compatible. Hence, $P$ has property K. 
\end{proof}

We want to show that 
\begin{displaymath}
V[\mc G]\models \forall {\alpha}\in {\mu}\
\forall J\in {[{\mu}\setm \{{\alpha}\}]}^{{\omega}}\ A_{\alpha}\setm \bigcup_{{\beta}\in J}A_{\beta}\ne \empt.
\end{displaymath}
Assume that $p\in P$ and 
\begin{displaymath}
p\Vdash  \dot{\alpha}\in {\mu}\land 
   \dot J\in {[{\mu}\setm \{{\alpha}\}]}^{{\omega}}.
  \end{displaymath}
  We can assume that $p$ decides the value of $\dot {\alpha}$.
  Since $P$ is c.c.c there is $K\in  {[{\mu}\setm \{{\alpha}\}]}^{{\omega}}\cap V$
  such that 
  \begin{displaymath}
  p\Vdash \dot J\subs K.
  \end{displaymath}
Pick ${\zeta}\in L$ with ${\zeta}>tp(K)$.
Pick ${\delta}\in {\omega}_1\setm D^p$ with ${\zeta}({\delta})={\zeta}$.
Define $q\in P$ as follows:
\begin{enumerate}[(a)]
\item $D^q=D^p\cup \{{\delta}\}$,
\item $E^q_{\eta}=E^p_{\eta}$ and  $F^q_{\eta}=F^p_{\eta}$ for ${\eta}\in D^p$,
\item $E^q_{\delta}=\{{\alpha}\}$ and $F^q_{\delta}=K$.
\end{enumerate}
Then, $q\in P$, $q\le p$ and 
\begin{displaymath}
q\Vdash {\delta}\in A_{\alpha}\setm \bigcup\{A_{\beta}:{\beta}\in K\}.
\end{displaymath}
Thus, 
\begin{displaymath}
V[\mc G]\models A_{\alpha}\setm \bigcup\{A_{\beta}:{\beta}\in J\}\ne \empt.
\end{displaymath}
So the family $\{A_{\alpha}:{\alpha}<{\mu}\}$ really witnesses that
$\indn{{\omega}_1}{{\omega}_1}>{\omega}_3$.

Thus, we proved Lemma \ref{lm:Ilarge}
\end{proof}   

Now, we obtain  Theorem \ref{tm:con} from Lemmas \ref{lm:Ilarge} and 
\ref{lm:Ainccc}.
\end{proof}

\section{A forcing construction of an LLSP-space of height ${\omega}+1$}

Our aim is to prove the following result.

\vspace{2mm}
\begin{theorem}\label{Theorem 1}
 %{\bf Theorem 1}. 
 If $2^{\omega_1} = \omega_2$ then there is an $\omega_2$-closed and $\omega_3$-c.c. partial order ${\mathbb   P}$ such that in $V^{{\mathbb   P}}$ we have that $\langle \omega_2\rangle_{\om} \concat \langle \omega_3\rangle \in \paseq {\omega + 1}$.
\end{theorem}

 \vspace{2mm} So, we obtain as an immediate consequence of Theorem \ref{Theorem 1} that it is consistent that $2^{\omega_1} = \omega_2$ and $\langle \omega_2\rangle_{\om} \concat \langle \omega_3\rangle \in \paseq {\omega + 1}$.

%\vspace{1mm} The following proposition is easy to verify.

\begin{definition}\label{df:Definition 4}
  %\vspace{2mm} {\bf Definition 4}. 
  Let $f = \langle \kappa_{\alpha} : \al < \delta \rangle$ be a sequence of uncountable cardinals where $\delta$ is a non-zero ordinal. Put $T = \un \{T_{\al} : \al < \delta \}$ where each $T_{\al} = \{\al\}\times \ka_{\al}$. Let $P_f$ be the set of all $p = \langle X_p,U_p,D_p\rangle$ such that the following conditions hold:

  \begin{enumerate}[(i)]
  \item $X_p\subset T$ such that $|X_p\cap T_{\al}| < \ka_{\al}$ for all $\al < \delta$,

  \item $U_p:X_p\to \mathcal P(X_p)$ is a weak LLSP-cone assignment with witnessing partition $\langle X_p\cap T_{\al} : \al < \delta \rangle$,

  \item $D_p : [X_p]^2 \to [X_p]^{\om}$ is the LLSP-function associated with $U_p$.
  \end{enumerate}

  \vspace{1mm}
  If $p,q\in P_f$, we put $q\leq p$ if and only if the following holds:

  \begin{enumerate}[(a)]
  \item $X_q \supset X_p$,

  \item if $s\in X_p$ then $U_p(s)\subset U_q(s)$ and $U_q(s)\setminus U_p(s)\subset X_q\setminus X_p$,

  \item if $\{s,t\}\in [X_p]^2$ then $D_p(s,t) = D_q(s,t)$.
  \end{enumerate}
For $A\subs X_p$, write $U_q[A]=\bigcup \{U_q(a):a\in A\}$.
\end{definition}

\begin{proposition}\label{Proposition 5}
  Let $f = \langle \kappa_{\alpha} : \al < \delta \rangle$ be a non-empty sequence of uncountable regular cardinals. Assume $P\subs P_f$ and let 
  $\mbb P=\<P,\le\restriction P\>$.
  If  $\mbb P$ is ${\omega}_1$-closed and forcing with $\mbb P$
  preserves cofinalities and cardinals, moreover  
\begin{enumerate}[(C1)]
\item For every $s\in T$, the set 
$D_s = \{q\in P: s\in X_q \}$ is dense in ${\mathbb  P}$.
\item For each $t\in T$,  $A\in {[T\setm \{t\}]}^{{\omega}}$, 
${\beta}<\pi(t)$ and 
${\xi}<{\kappa}_{\beta}$, the set 
\begin{multline*}
D_{t,A,{\beta},{\xi}}=\{q\in P:A\subs U_q(t) \to \big(U_q(t)\setm 
U_q[A]\big)\cap \{{\beta}\}\times ({\kappa}_{\beta}\setm {\xi}) \ne \empt    \}
\end{multline*} 
is dense in $\mathbb P$.
\end{enumerate}
Then, forcing with ${\mathbb   P}$ adjoins an LLSP space whose cardinal sequence is $f$.
\end{proposition}

\begin{proof}
 Let $G$ be a ${\mathbb   P}$-generic filter. 
 Define $\<X,U,D\>$ as follows:
\begin{enumerate}[(i)]
\item $X=\bigcup\{X_p:p\in \mc G\}$, 
\item $U:{X}\to \mc P(X)$ such that 
$U(x)=\bigcup\{U_p(x):  x\in X_p\}$,
\item $D:{[X]}^{2}\to {[X]}^{{\omega}}$ such that 
$D\{x,y\}=D_p\{x,y\}$  for  each $\{x,y\}\in \dom D_p$.
\end{enumerate} 
By $(C1)$, we see that $X = T$. 
 Note that by condition $(C2)$, we have that for all $\alpha < \beta < \delta$, for all $t\in T_{\beta}$ and for all $I\in [U(t)\setminus \{t\}]^{\omega}$ 
 the set 
 $\{{\xi}<{\kappa}_{\alpha}:\<{\alpha},{\xi}\>\in U(t)\setm U[I]\}$
 is cofinal in ${\kappa}_{\alpha}$. Since ${\kappa}_{\alpha}$ 
 is regular even in the generic extension, 
  $|(U(t)\setminus U[I])\cap T_{\alpha}| = \kappa_{\alpha}$. Hence, $U$ is an LLSP-cone assignment for $X$, and so $X_U$  is an LLSP space such that $I_{\alpha}(X_U) = T_{\alpha}$ for every $\alpha < \delta$. Hence, $\mbox{CS}(X_U) = f$. 
\end{proof}

\begin{proof}[Proof of Theorem \ref{Theorem 1}] Consider the partial order ${\mathbb   P}_f$ where $f = \langle \omega_2\rangle_{\om} \concat \langle \omega_3\rangle$. Then, we define the partial order ${\mathbb   P} = (P,\leq\restriction P)$ as follows:
 \begin{displaymath}
 P=\{p\in P_f: |X_p\cap T_{\alpha}|={\omega}_1\text{ for each }{\alpha}\le {\omega}\}.
 \end{displaymath} 
Our plan is to apply Proposition \ref{Proposition 5} for $\mathbb   P$.

  \vspace{1mm} Clearly, ${\mathbb   P}$ is $\om_2$-closed. 

The main burden of the proof is the following lemma.   
\begin{lemma}\label{lm:chain}
${\mathbb   P}$ has the $\om_3$-chain condition.
\end{lemma}

\begin{proof}

  In order to prove that, assume that $R = \{p_{\nu} : \nu < \om_3 \} \subset P$. Write $p_{\nu} = \langle X_{\nu}, U_{\nu}, D_{\nu} \rangle$ for $\nu < \om_3$. By thinning out $R$ if necessary, we may assume that $\{X_{\nu} : \nu < \om_3 \}$ forms a $\Delta$-system with kernel $X^*$ in such a way that for every $n < \om$, $X_{\nu}\cap T_n = X^* \cap T_n$. Let $Z_{\nu} = X_{\nu}\cap T_{\om}$ for $\nu < \om_3$. We consider in $T_{\om}$ the order induced by $\om_3$. By thinning out $R$ again if necessary, we may assume that there is an ordinal $\zeta^* < \om_2$ such that the order type of $Z_{\nu}\setminus X^*$ is $\zeta^*$ for every $\nu <\om_3$. Also, we may suppose that for $\nu < \mu < \om_3$ there is a bijection $h_{\nu\mu} : X_{\nu} \to X_{\mu}$ satisfying the following conditions:

  \begin{enumerate}[(a)]
  \item  $h_{\nu\mu}$ is the identity on $X^*$.

  \item For every $\xi < \zeta^*$, if $s$ is the $\xi$-element in $Z_{\nu}\setminus X^*$ and $t$ is the $\xi$-element in $Z_{\mu}\setminus X^*$  then $h_{\nu\mu}(s) = t$.

  \item $s\in U_{\nu}(t)$ iff $h_{\nu\mu}(s)\in U_{\mu}(h_{\nu\mu}(t))$ for all $\{s,t\}\in [X_{\nu}]^2$.

  \item $D_{\mu}(h_{\nu\mu}(s),h_{\nu\mu}(t)) = D_{\nu}(s,t)$ for all $\{s,t\}\in [X_{\nu}]^2$.

  \end{enumerate}

  Note that we deduce from $(a)$ and $(c)$ that 
  \begin{enumerate}[(a)]\addtocounter{enumi}{4}
  \item 
  if $t\in X^*$ then $U_{\nu}(t) = U_{\mu}(t)$.
  \end{enumerate}
   And we deduce from $(a)$ and $(d)$ that 
   \begin{enumerate}[(a)]\addtocounter{enumi}{5}
    \item 
   if $\{s,t\}\in [X^*]^2$ then $D_{\nu}(s,t) = D_{\mu}(s,t)$. 
\end{enumerate}

  Fix $\nu < \mu < \om_3$. The following Claim clearly completes the proof of the Lemma. 
  
  \smallskip
  \noindent{\bf Claim.} {\em  There is a  condition $r=\langle X_r, U_r,D_r \rangle \in P$ such that $r\leq p_{\nu},p_{\mu}$.} 
  \smallskip

  We need some preparation. 
  Put ${\tau} = {\tau}_{U_{\nu}}$. 
  By Theorem \ref{tm:llspgen}(1), ${\tau}$ is an LLSP topology on $X_{\nu}$.
From now on when we write ``{\em Lindelöf}'' we mean `{\em Lindelöf in the topology ${\tau}$}''.

 \noindent{\em Remark.} Note that if $s\in (X_{\nu}\setminus X_{\mu})\cap T_{\om}$ and $t=h_{{\nu}{\mu}}(s)\in (X_{\mu}\setminus X_{\nu})\cap T_{\om}$, 
  it may happen that 
  $U_{\nu}(s)\cap U_{\mu}(t)=U_{\nu}(s)\setm \{s\}$ is not Lindelöf, so there is no  $I\in [U_{\nu}(s)\cap U_{\mu}(t)]^{\omega}$
  with $U_{\nu}(s)\cap U_{\mu}(t)\subs U_{\nu}[I]$. 
  Thus, in order to amalgamate $p_{\nu}$ and $p_{\mu}$, we will need to add new elements to $X_{\nu}\cup X_{\mu}$.

Fix an enumeration  $\{y_{\xi}:{\xi}<{\omega}_1\}$  of $Z_{\xi}$, and 
for ${\xi}<{\omega}_1$ let 
 \begin{displaymath}
 V_{\xi}=U_{\nu}(y_{\xi})\setm U_{\nu}[\{y_{\zeta}:{\zeta}<{\xi}\}].
 \end{displaymath} 
Next, consider an enumeration $\{z^{\xi}_i:i<{\vartheta}_{\xi}\}$ of 
$V_{\xi}\setm \{y_{\xi}\}$, where ${\vartheta}_{\xi}\le {\omega}_1$, 
and let 
\begin{displaymath}
\mc W_{\xi}=\{V_{\xi}\cap \big(U_{\nu}(z^{\xi}_j)
\setm U_{\nu}[\{z^{\xi}_i:i<j\}]\big):j<{\vartheta}_{\xi}\}\setm \{\empt\}.
\end{displaymath}
Then, $\mc W_{\xi}$ is a partition of $V_{\xi}\setm \{y_{\xi}\}$ into non-empty clopen Lindelöf
subsets. Moreover,  $|\mc W_{\xi}|\le {\vartheta}_{\xi}\le {\omega}_1$, and 
for each $W\in \mc W_{\xi}$ there is a positive natural number $n_{\xi}(W)$
such that 
\begin{displaymath}
W\subs \bigcup\{T_n:n < n_{\xi}(W)\}.
\end{displaymath}
For $1\le n<{\omega}$ let   
\begin{displaymath}
  \mc W_{{\xi},n}=\{W\in \mc W_{\xi}:n_{\xi}(W)=n\}.
\end{displaymath}
Let $K=\{{\xi}<{\omega}_1: y_{\xi}\notin X^*\}$, and 
for each $1\le n<{\omega}$ pick pairwise different elements 
\begin{displaymath}
\{w^{\xi}_n:{\xi}\in K\}\subs T_n\setm (X_{\nu}\cup X_{\mu}).
\end{displaymath}
We start to define $r$.
For $y\in X_{\nu}$, write  $\bar y=h_{{\nu}{\mu}}(y)$.
Let 
\begin{enumerate}[(1)]
\item $X_r=X_{\nu}\cup X_{\mu}\cup\{w^{\xi}_n:{\xi}\in K, n<{\omega}\},\smallskip$
\item for $s\in X_r$ let 
\begin{displaymath}
U_r(s){}=\left\{\begin{array}{ll}
{U_{\nu}(s)}&\text{if $s\in X^*$,}\medskip\\
{U_{\nu}(y_{\xi})\cup\{w^{\xi}_n:n<{\omega}\}}&\text{if $s=y_{\xi}\in X_{\nu}\setm X^*$,}\medskip\\
{U_{\mu}(y_{\xi})\cup\{w^{\xi}_n:n<{\omega}\}}&\text{if $s=\bar y_{\xi}\in X_{\mu}\setm X^*$,}\medskip\\
\{w^{\xi}_n\}\cup\bigcup\mc W_{{\xi},n}&\text{if $s=w^{\xi}_n$.}\end{array}\right.
\end{displaymath}
\end{enumerate}
So $r$ satisfies \ref{df:llsp-cone}\eqref{cone:a}.

If $S\subs X_{\nu}$ is Lindelöf, then $\{U_{\nu}(s):s\in S\}$ is an open cover of $S$, so we can fix  $D(S)\in {[S]}^{{\omega}}$ such that $S\subs U_{\nu}[D(S)]$.

If $s\in X_{\nu}\cap T_n$ for some $n<{\omega}$ and ${\xi}<{\omega}_1$ then $U_{\nu}(s)\cap V_{\xi}$ is Lindelöf. 
Moreover, writing $$\mc W_{\xi}(s)=\{W\cap U_{\nu}(s):W\in \mc W_{\xi}\}\setm \{\empt\}$$ we have 
\begin{displaymath}
  U_{\nu}(s)\cap V_{\xi}=U_{\nu}(s)\cap (V_{\xi}\setm \{y_{\xi}\})=\bigcup \mc W_{\xi}(s), 
\end{displaymath} 
i.e. $\mc W_{\xi}(s)$ is a partition of a Lindelöf space into open Lindelöf subspaces. Therefore, 
\begin{displaymath}
\text{$\mc W_{\xi}(s)$ is countable.}
\end{displaymath} 
Next we define $D_r$ and in parallel we verify that $D_r$ satisfies 
\ref{df:llsp-cone}\eqref{cone:c}-\eqref{cone:d}.

\vspace{2mm}\noindent {\bf  Case 1}. $\{s,t\}\in {[X_{\nu}]}^{2}$.

Let 
\begin{equation}\label{Case1}
D_r(s,t)=D_{\nu}(s,t).
\end{equation}
Since $U_r(s)\cap U_{r}(t)=U_{\nu}(s)\cap U_{\nu}(s)$,
and $U_r(s)\setm  U_{r}(t)=U_{\nu}(s)\setm  U_{\nu}(s)$  
if $s\in U_{\nu}(t)$,
definition \eqref{Case1} works.  

For $\{s,t\}\in {[X_{\mu}]}^{2}$, the set $D_r(s,t)$ can be defined analogously:
let $D_r(s,t)=D_{\mu}(s,t)$.

\vspace{2mm}\noindent {\bf  Case 2}. $s\in X_{\nu}\setm X^*$,
$t\in X_{\mu}\setm X^*$, $t\ne \bar s$.

Let 
\begin{equation}\label{Case2}
D_r(s,t)=D_{\mu}(\bar s,t).
\end{equation}
Since $U_r(s)\cap U_r(t)=U_{\mu}(\bar s)\cap U_{\mu}(t)$,
definition \eqref{Case2} works.  

\vspace{2mm}\noindent {\bf  Case 3}. $s=y_{\xi}\in X_{\nu}\setm X^*$,
$t\in X_{\mu}\setm X^*$, $t= \bar s$.

Let 
\begin{equation}\label{Case3}
D_r(s,t)=\{w^{\xi}_n:1\le n<{\omega}\}.
\end{equation}
Then
\begin{multline*}
  U_r(s)\cap U_r(t)=U_r(y_{\xi})\cap U_r(\bar y_{\xi})=(V_{\xi}\setm \{y_{\xi}\})\cup
  \{w^{\xi}_n:1\le n<{\omega}\}=\\\bigcup \mc W_{\xi}\cup
  \{w^{\xi}_n:1\le n<{\omega}\}=
  \bigcup_{1\le n <{\omega}}\big(\{w^{\xi}_n\}\cup\bigcup \mc W_{{\xi},n}\big)\subs U_r[\{w^{\xi}_n:1\le n<{\omega}\}]. 
  \end{multline*}
Thus,   definition \eqref{Case3} works.

\vspace{2mm}\noindent {\bf  Case 4}. $\{s,t\}=\{w^{\xi}_n,w^{\zeta}_m\}$. 

Let 
\begin{equation}\label{Case4}
  D_r\{w^{\xi}_n,w^{\zeta}_m\}=\empt.
  \end{equation}
  
If ${\zeta}\ne {\xi}$, then $U_r(w^{\xi}_n)\cap U_r(w^{\zeta}_m)\subs V_{\xi}\cap V_{\zeta}=\empt$. 
If ${\xi}={\zeta}$, then  $n\ne m$, and so   
$U_r(w^{\xi}_n)\cap U_r(w^{\zeta}_m)=(\bigcup \mc W_{{\xi},m})\cap
(\bigcup \mc W_{{\xi},n})= \empt$ because 
$\mc W_{{\xi},m}\cap \mc W_{{\xi},n}=\empt$, and   the elements of $\mc W_{\xi}$ are pairwise disjoint.
So,  in both cases, definition \eqref{Case4}
works.

\vspace{2mm}\noindent {\bf  Case 5}. {\em $s=w^{\xi}_n$ and   $t\in X_{\nu}$.} 

Assume first that $t=y_{\xi}$.
Then  $U_r(s)\subs U_r(t)$. Thus, the setting
\begin{equation}\label{Case5a}
  D_r\{s,t\}=\empt
  \end{equation}
works. 

Assume now that  $t\ne y_{\xi}$.
Thus, $s\notin U_r(t)$.
Then, $V_{\xi}\cap U_r(t)\subs X^*$ is  a clopen Lindelöf subset of $V_{\xi}\setm \{y_{\xi}\}$.  
Since $\mc W_{\xi}$ is a clopen partition of $V_{\xi}\setm \{y_{\xi}\}$,
\begin{displaymath}
  |\{W\in \mc W_{\xi}:W\cap U_r(t)\ne \empt\}|\le {\omega}. 
  \end{displaymath}
  Since $U_r(s)\cap U_r(t)\subs V_{\xi}\cap U_{\nu}(t)$, the family 
\begin{displaymath}
  \mc W=\{W\in \mc W_{\xi}:W\cap U_r(s)\cap U_r(t)\ne \empt\}
  \end{displaymath}
  is countable. 
  If $W\in \mc W$, then $W\subs U_r(s)$, so 
  $W\cap U_r(s)\cap U_r(t)=W\cap U_r(t)$, and so it is clopen Lindelöf.
  Thus $U_r(s)\cap U_r(t)$ is the union of countably many clopen Lindelöf subspaces, 
  so it is also closed Lindelöf.

  Now we have to consider the following two subcases: (1) $t\not\in U_r(s)$, and (2) $t\in U_r(s)$.

So if  $s\notin U_r(t)$, then the setting 
\begin{equation}\label{Case5b}
  D_r\{s,t\}=D(U_r(s)\cap U_r(t))
  \end{equation}
works. 

% If $s\in U_r(t)$, then $U_r(s)\setm U_r(t)$ is Lindelöf, because 
% both $U_r(s)$  is Lindelöf  $U_r(t)\cap U_r(t)$ is open. 
% Thus, 
% the  setting 
% \begin{equation}\label{Case5c}
%   D_r\{s,t\}=D(U_r(s)\setm  U_r(t))
%   \end{equation}
% works. 

%\hrule
Now, assume that $t\in U_r(s)$. So, $t\in X^*$ and $U_r(t) = U_{\nu}(t)$.  Since $U_{\nu}(t)\cap V_{\xi}$ is Lindel\"{o}f, we have that
$\{W\in \mc W_{{\xi},n}: W \cap U_{\nu}(t) \neq \emptyset \}$ is countable, and so $U_{\nu}(t)\setminus \bigcup \mc W_{{\xi},n}$ is Lindel\"{o}f. Then, $U_r(t)\setminus U_r(s) \subset U_{\nu}(D(U_{\nu}(t)\setminus \bigcup \mc W_{{\xi},n}))$. Thus, we define
\begin{equation}\label{eq:5c}
  D_r(s,t) = D(U_{\nu}(t)\setminus \bigcup \mc W_{{\xi},n})
\end{equation}

\vspace{2mm}\noindent {\bf  Case 6}. {\em $s=w^{\xi}_n$ and   $t\in X_{\mu}$.} 

We can proceed as in Case 5.
%\hrule 
\medskip

Hence, we proved the claim,  which completed the proof of the Lemma.  \end{proof}

  \vspace{2mm} So, forcing with ${\mathbb   P}$ preserves cardinals and cofinalities. Clearly, condition (C1) of Proposition \ref{Proposition 5} holds. Now, we prove condition (C2). So assume 
  $t\in T$,  $A\in {[T\setm \{t\}]}^{{\omega}}$, 
${\beta}<\pi(t)$ and 
${\xi}<{\kappa}_{\beta}$.  

Since ${\mathbb   P}$ is $\omega_2$-closed, we have that $D_{t,A,{\beta},\xi}\in V$. Then, consider $p =\langle X_p,U_p,D_p\rangle\in P$. We define $q\in D_{t,A,{\beta},\xi}$ such that $q\leq p$. Without loss of generality, we may assume that $t\in X_p$. 

If $A\not\subset U_p(t)$, then $q=p$ satisfies the requirements. 
So we can suppose that $A\subs U_p(t)$.

% First, suppose that $t_n\not\in X_p$ for some $n\in \omega$. Then, we define $q=\langle X_q,U_q,D_q\rangle$ as follows:

%   \begin{enumerate}[(a)]
%   \item $X_q = X_p \cup \{t_n : n\in \omega \}$,

%   \item $U_q(x) = U_p(x)$ for all $x\in X_p$, and $U_q(x) = \{x\}$ for all $x\in X_q\setminus X_p$,

%   \item $D_q(x,y) = D_p(x,y)$ if $\{x,y\}\in [X_p]^2$, $D_q(x,y) = \emptyset$ otherwise.

%   \end{enumerate}

%   \vspace{1mm} Now, suppose that $t_n\in X_p$ for every $n\in \omega$.  

%If $t_n\not\in U_p(t)$ for some $n\in \omega$, we put $q = p$. So, assume that $t_n\in U_p(t)$ for all $n\in \omega$. 

Let $u\in T_k\setminus X_p$ be such that $u =\langle k,\zeta\rangle$ with $\zeta > \xi$. We define $q=\langle X_q,U_q,D_q\rangle$ as follows:

  \begin{enumerate}[(a)]
  \item $X_q = X_p \cup \{u \}$,

  \item $U_q(x) = U_p(x)$ if $x\in X_p$ and $t\not\in U_p(x)$, $U_q(x) = U_p(x)\cup \{u\}$ if $t\in U_p(x)$, and $U_q(u) = \{u\}$.

  \item $D_q(x,y) = D_p(x,y)$ for all $\{x,y\}\in [X_p]^2$, and  $D_q(x,u) = \emptyset$ for all $x\in X_p$.

  \end{enumerate}

  \vspace{1mm} Clearly, $q$ is as required.

So we can apply Proposition \ref{Proposition 5} to complete the proof of Theorem \ref{Theorem 1}.
\end{proof}

 Now, consider a $\mbb P$-generic filter $G$, where $\mbb P$ is the notion of
  forcing defined in Theorem \ref{Theorem 1}. We work in $V[G]$. Let $X$ be an LLSP
   space such that 
   $\mbox{CS}(X) = \langle \omega_2\rangle_{\omega}\concat \langle \omega_3 \rangle$. By \cite[Theorem 2.1]{MaSoLinSP22}, there is 
   an LLSP space $Y$ such that $\mbox{CS}(Y) = \langle \omega_1\rangle_{\omega_2}$. 
We may assume that the underlying sets of $X$ and $Y$ are disjoint. Let $Z$ be the
 topological sum  of $X$ and $Y$. So, $\mbox{CS}(Z) = \langle \lambda_{\al} : \al < 
\omega_2 \rangle$ where $\lambda_n = \omega_2$ for 
$n < \omega$, $\lambda_{\omega} = \omega_3$ and $\lambda_{\alpha} = \omega$ 
for $\omega < \lambda  < \omega_2$. 
Since $\omega_2^{\omega_1} = \omega_2$, it is easy to see that 
$\indn{{\omega_2}}{{\omega}_1} = \omega_3$, and so in 
Theorem \ref{tm:positive}(2) we can not replace 
${\kappa}_{\beta}< \indn {{\kappa}_{\alpha}}{{\kappa}_{\alpha}}$ with 
${\kappa}_{\beta}< \indn {{\kappa}_{\alpha}}{{\omega}_1}$. Also, since   
$\almd{{\omega_2}}{{\omega_1}}\leq \indn{{\omega_2}}{{\omega}_1}$ and 
$\omega_2^{\omega} = \omega_2$, 
we have that $\almd{{\omega_2}}{{\omega_1}} = \omega_3$, and hence the reverse 
implication of Theorem \ref{tm:positive}(1) does not hold.

\end{document}